\documentclass[11pt]{amsart}
\usepackage{mathrsfs,latexsym,amsfonts,amssymb}
\usepackage{hyperref}
\newtheorem{theorem}{Theorem}[section]
\newtheorem{proposition}[theorem]{Proposition}
\newtheorem{lemma}[theorem]{Lemma}
\newtheorem{corollary}[theorem]{Corollary}
\newtheorem{question}[theorem]{Question}

\newtheorem{example}[theorem]{Example}
\theoremstyle{remark}
\newtheorem{remark}[theorem]{Remark}

\begin{document}
\title[Suitable sets for paratopological groups]
{Suitable sets for paratopological groups}

\author{Fucai Lin*}
\address{Fucai Lin: 1. School of mathematics and statistics,
Minnan Normal University, Zhangzhou 363000, P. R. China; 2. Fujian Key Laboratory of Granular Computing and Application, Minnan Normal University, Zhangzhou 363000, China}
\email{linfucai@mnnu.edu.cn; linfucai2008@aliyun.com}

\author{Alex Ravsky}
\address{Alex Ravsky: Pidstryhach Institute for Applied Problems of Mechanics and Mathematics,
3b Naukova str., 79060, Lviv, Ukraine}
\email{alexander.ravsky@uni-wuerzburg.de}

\author{Tingting Shi}
\address{Tingting Shi: School of mathematics and statistics,
Minnan Normal University, Zhangzhou 363000, P. R. China}
\email{277653220@qq.com}

\thanks{The first author is supported by the Key Program of the Natural Science Foundation of Fujian Province (No: 2020J02043), the NSFC (No. 11571158), the lab of Granular Computing, the Institute of Meteorological Big Data-Digital Fujian and Fujian Key Laboratory of Data Science and Statistics.}
\thanks{* Corresponding author}


\date{\today}
\begin{abstract}
A paratopological group $G$ has a {\it suitable set} $S$.
The latter means that $S$ is a discrete subspace of $G$, $S\cup \{e\}$ is closed, and
the subgroup $\langle S\rangle$ of $G$ generated by $S$ is dense in $G$.
Suitable sets in topological groups were studied by many authors.
The aim of the present paper is to provide a start-up for a general investigation
of suitable sets for paratopological groups,
looking to what extent we can (by proving propositions)
or cannot (by constructing examples)
generalize to paratopological groups results which hold for
topological groups, and to pose a few challenging questions for possible future research.
We shall discuss when paratopological groups of
different classes have suitable sets. Namely,
we consider paratopological groups (in particular, countable)
satisfying different separation axioms,
paratopological groups which are compact-like spaces,
and saturated (in particular, precompact) paratopological groups.
Also we consider the permanence of a property of a group to have a suitable set
with respect to (open or dense) subgroups, products and extensions.
\end{abstract}

\maketitle

\section{Introduction}

A {\em paratopological group} is a group $G$ endowed with a topology $\tau$
such that the group operation $\cdot:G\times G\to G$ is continuous.
In this case $\tau$ is called a {\it semigroup topology} on $G$.
If, additionally, the operation of taking inverse is continuous
then $G$ is a {\it topological group}.
A classical example of a paratopological group failing to be a
topological group is the {\it Sorgenfrey line} $\mathbb S$, that is the
additive group of real numbers, endowed with the Sorgenfrey topology
generated by the base consisting of half-intervals $[a,b)$, $a<b$.

Whereas the investigation of topological groups already is a fundamental branch of topological
algebra (see, for instance,~\cite{Pon},~\cite{DikProSto}, and~\cite{AT}),
paratopological groups are not so well-studied and have more variable structure.
Basic properties of paratopological groups compared with
the properties of topological groups are described in the book \cite{AT} by Arhangel'skii
and Tkachenko, in the PhD thesis of the second author~\cite{Rav3},
the papers~\cite{Rav}, ~\cite{Rav2}, and in the survey \cite{Tka5} by Tkachenko.

Suitable sets were considered in the context of Galois cohomology by Tate (see~\cite{Doa}) and in
the context of free profinite groups by Mel'nikov~\cite{Mel}. Later Hofmann
and Morris promoted the concept of suitable set in their seminal paper~\cite{HM1990} and in their
well-known monograph on compact groups.
Since then many authors studied suitable sets in topological groups,
see, for instance, ~\cite{CG-F}, \cite{Tka97}.
Fundamental results were obtained by Comfort et al. in~\cite{CMRS1998}
and Dikranjan et al. in~\cite{DTT1999} and in~\cite{DTT2000}.
Many examples of topological groups without suitable sets were constructed in~\cite{Tka97}.

As far as we know, the first few results on suitable sets for paratopological groups
were obtained by Guran, see~\cite{G2003} and Theorem 13.3 of ~\cite{BP2003}.
The aim of the present paper is to provide a start-up for a general investigation
of this topic, looking to what extent we can (by proving propositions)
or cannot (by constructing examples)
generalize to paratopological groups results which hold for
topological groups, and to pose a few challenging questions for possible future research.

A subset $S$ of a paratopological group $G$ is a {\it suitable set} for $G$, if
$S$ is a discrete subspace of $G$, $S\cup \{e\}$ is closed, and
the subgroup $\langle S\rangle$ of $G$ generated by $S$ is dense in $G$, see~\cite{G2003}.
Let $\mathcal{S}$ (respectively, $\mathcal{S}_{c}$) be the class of paratopological groups $G$
having a suitable (respectively, closed suitable) set.
It turns out that very often a suitable set of a group $G$ generates $G$.
This fact suggests to devote a special attention to classes
$\mathcal{S}_g$ (respectively, $\mathcal{S}_{cg}$) of paratopological groups $G$
having a suitable (respectively, closed suitable) set which generates $G$.

In the paper we shall discuss when paratopological groups of
different classes have suitable sets.
Whereas the constructed examples are usually specific
for paratopological groups, a lot of presented positive results are counterparts
of those from papers ~\cite{CMRS1998} and~\cite{DTT1999}.
Nevertheless Propositions~\ref{prop:CCnotCS}, \ref{p0new}, \ref{t2}, \ref{t2.2},
and \ref{t2.3} can be new for topological groups.

Section~\ref{sec:sep} is devoted to paratopological groups (in particular, to countable)
satisfying different separation axioms.
Generalizing Proposition 1.4 from~\cite{CMRS1998},
in Proposition~\ref{tt} we show that any paratopological group with a suitable set
is a $T_1$-space or a two-element group.
In Examples~\ref{ex:ZT1notT2} and~\ref{ex:T1nS}
are provided infinite $T_1$ non-Hausdorff paratopological groups
with and without a suitable set, respectively.

Section~\ref{sec:compactlike} is devoted to paratopological groups which are compact-like spaces.
We observe that Theorem 3.2 from~\cite{DTT1999} can be generalized to paratopological groups.
That is, if $\kappa$ is an infinite non-measurable cardinal and $\lambda=2^{\kappa}$ and
$G$ is a $T_1$ paratopological group such that $d(G)\le\lambda$ and
the group $G^\lambda$ is not countably compact then $G^\lambda$ has a closed suitable set.
In Proposition~\ref{prop:CCnotCS} we show
that if $G$ is a countably compact infinite locally finite $T_1$ paratopological group without non-trivial convergent
sequences then $G$ has no suitable set.
In Example~\ref{e1} we show that a non-Hausdorff $T_1$ periodic countable sequentially pracompact
Abelian paratopological group constructed by Banakh in~\cite[Example 3.18]{BR2020}
belongs to $\mathcal{S}_{cg}$.
Complementing Proposition~\ref{prop:CCnotCS}, in Example~\ref{e2} we show that
a feebly compact non-countably compact Baire Hausdorff paratopological group
such that all its countable subsets are closed,
constructed by Sanchis and Tkachenko in the proof of~\cite[Theorem
2]{ST2012}, belongs to $\mathcal{S}_{cg}$.
Extending Proposition 2.7 from~\cite{DTT1999}, in
Proposition~\ref{p0new} we show that if $G$ is a countably compact $T_1$ paratopological group
with a suitable set, $H$ is a Hausdorff paratopological group, and $f:G\to H$ be a
continuous homomorphism with dense image then $H$ has a suitable set.

Section~\ref{sec:saturated} is devoted to saturated (in particular, to precompact)
paratopological groups. In Proposition~\ref{t4} we show
that if $G$ is a saturated Hausdorff paratopological group and $S$ be a (closed) suitable
set for a the group reflexion $G^{\flat}$ of $G$ then $S$ is a (closed) suitable set for a group
$G$ too.
Guran in~\cite{G2003} announced that every Hausdorff separable non-precompact
paratopological group $G$ has a closed suitable set.
Complementing this, in Proposition~\ref{l15+} we show that any
non-feebly compact precompact separable $T_1$ paratopological group
has a closed suitable set.
Also we show that each Hausdorff locally separable saturated first countable
(non-pre\-com\-pact) paratopological group has a (closed) suitable set,
see Corollary~\ref{c5+}.
In Example~\ref{ex:GflatwoSS} we provide
a zero-dimensional saturated $T_1$ paratopological group $G$
which is a sum of two of its closed discrete subsets such that the group reflexion
$G^{\flat}$ has no suitable set.
Adapting Theorem 3.6.I from~\cite{DTT1999},
in Proposition~\ref{ttt+} we show that
every non-feebly compact non-precompact saturated $T_1$ paratopological group $G$ with
a dense strictly $\sigma$-discrete subspace has a closed suitable set.
In Proposition~\ref{qdpart+} we show that every non-precompact saturated $T_1$ paratopological group
which is a $\sigma$-space has a suitable set.

Section~\ref{sec:perm} is devoted to
the permanence of a property of a group to have a suitable set
with respect to (open or dense) subgroups, products and extensions.
We observe that Theorem 4.2 from~\cite{CMRS1998} can be extended to paratopological groups,
that is if an open subgroup of a $T_1$ paratopological group $G$
has a (closed) suitable set, then $G$ has a (closed) suitable set.
In Section~\ref{sec:perm2}
we show, in particular, that if a collectionwise Hausdorff paratopological group $G$ has a suitable set $S$ and
$H$ is a (sequentially) dense subgroup of $G$ then $H$ has a suitable set of size at most
$|S|\cdot\aleph_0$, if $G$ is hereditarily collectionwise Hausdorff, see Proposition~\ref{t2}
or regular with countable pseudocharater, see Proposition~\ref{t2.3}.
In Section~\ref{sec:perm3} we observe the following. Theorem 4.3 from~\cite{CMRS1998}
can be generalized to paratopological groups as follows.
Let $\Gamma$ be a non-empty family of $T_1$ paratopological groups
such that each member of it has a suitable set. Then the product of $\Gamma$
has a suitable set contained in the $\sigma$-product of $\Gamma$. It follows
that both the $\sigma$- and $\Sigma$-product of $\Gamma$ has a suitable set.
Lemma 3.1 from~\cite{DTT1999}
can be generalized to paratopological groups as follows.
If a $T_1$ paratopological group $G$ contains a closed discrete set $A$
such that $|A|\geq d(G)$ then $G\times G\in \mathcal{S}_{c}$.
Moreover, if $|A|=|G|$ then $G\times G\in \mathcal{S}_{cg}$.
Complementing Proposition~\ref{l1} and extending Lemma 2.3 from ~\cite{DTT1999} to
paratopological groups, in Proposition~\ref{prop:DPTGSS} we show that
if $G$ is a $T_1$ paratopological group with
a suitable set then $d(G)\leq \psi(G)e(G)$. Moreover,
if $G$ has a closed suitable set then $d(G)\le e(G)$.

\section{Separation axioms in paratopological groups}\label{sec:sep}

These axioms describe specific structural properties of spaces.
Basic separation axioms and relations between them are considered in~\cite[Section 1.5]{Eng}.

It is well-known that  each topological group is completely regular.
On the other hand, simple examples show that none of the implications
$T_0\Rightarrow T_1 \Rightarrow T_2 \Rightarrow T_3$ holds for paratopological groups
(see \cite[Examples 1.6-1.8]{Rav} and
page 5 in any of papers \cite{Rav2} or \cite{Tka5})
and there are only a few backwards implications between different separation axioms, see
~\cite[Section 1]{Rav2} or~\cite[Section 2]{Tka5}.
On the other hand, Banakh and Ravsky proved in~\cite{BR2016} that every $T_0$ (semi)regular
paratopological group is Tychonoff.

The following proposition generalizes Proposition 1.4 from~\cite{CMRS1998} to paratopological
groups.

\begin{proposition}\label{tt} Any paratopological group with a suitable set
is a $T_1$-space or a two-element group.
\end{proposition}
\begin{proof} Let $S$ be a suitable set of a paratopological group $G$ and $s$ be any element of $S$.
The definition of a suitable set implies $\overline{\{s\}}\subset \{s,e\}$.
If $\overline{\{s\}}$ or $\overline{\{e\}}$ is a one-point set then $G$ is a $T_1$-space.
Otherwise $\overline{\{s\}}=\overline{\{e\}}=\{s,e\}$. It follows $s^2\in \overline{\{s\}}$,
so $s^2=e$ and $G=\overline{\langle S\rangle}=\{s,e\}$.
\end{proof}

By Proposition~\ref{tt}, any topological group with at least three elements and a suitable set
is a Tychonoff space. On the other hand, the following simple example shows that
an infinite paratopological group with a suitable set can fail to be Hausdorff.

\begin{example}\label{ex:ZT1notT2} Let $\vec{\mathbb Z}$ be a group $\mathbb Z$
endowed with a base $\{U_{n,m}:n,m\in\mathbb Z\}$, where $U_{n,m}=
\{k\in\mathbb Z:k=n\mbox{ or } k\ge m\}$ for each $n,m\in\mathbb Z$.
Then $\vec{\mathbb Z}$ is a non-Hausdorff paratopological group with a suitable set $\{-1,0\}$.
\end{example}

{\it All spaces below are supposed to be $T_{1}$}, if the opposite is not stated.

According to~\cite[5.3]{BR2020}, topologies $\tau$ and $\sigma$ on a set $X$ are defined to be
{\it cowide} if for any
nonempty sets $U\in\tau$ and $V\in\sigma$ an intersection $U\cap V$ is nonempty. A
topology cowide to itself is called {\it wide}.
The spaces with wide topology are well-known as {\it irreducible} spaces,
see~\cite{AM},~\cite{mDon}.

A subset $U$ of a space is {\it regular open}, if $U=\operatorname{int}\overline{U}$.
Stone \cite{Sto} and Kat\v{e}tov \cite{Kat} considered the topology $\tau_{sr}$ on a
space $(X,\tau)$,
generated by a base consisting of all regular open sets of the space $(X,\tau)$. The space
$(X,\tau_{sr})$ is called the {\it semiregularization} of the space $(X,\tau)$.
It is easy to show that the semiregularization of a Hausdorff space is Hausdorff.
Moreover, a semiregularization of a Hausdorff paratopological group is a regular paratopological group,
see~\cite[Example 1.9]{Rav2}, \cite[p. 31]{Rav3} or \cite[p. 28]{Rav3}.

\begin{example}\label{ex:T1nS} There exists a non-Hausdorff paratopological group without a suitable set.
There exists a paratopological group generated by its closed discrete subset
such that the semiregularization of the group has no suitable set.
\end{example}
\begin{proof}
According to~\cite[Theorem 2.4.a]{DTT1999} there exists a non-separable Lindel\"of
topological linear space $(L,\tau)$ of countable pseudocharacter. The Cartesian product $G$ of $L$ and the group
$\vec{\mathbb Z}$ from Example~\ref{ex:ZT1notT2} is a non-separable Lindel\"of
paratopological group of countable
pseudocharacter. By Proposition~\ref{prop:DPTGSS}, $G$ has no suitable set.

We can consider $L$ as a linear space over a field $\mathbb Q$. Let $B$ be a basis of this space. Pick
any vector $e\in B$ and define a (unique) linear map $s:L\to\mathbb Q$ such that $s(e)=1$ and
$s(B\setminus\{e\})=0$. Let $T=\{x\in L:s(x)\ge 0\}$. Let $\sigma$ be a semigroup topology on
$L$ consisting of all sets $U\subset L$ such that for every $u\in U$
there exists $x\in T$ such that $u+x+T\subset U$, see~\cite[5.2]{BR2020}.
Let $\tau\vee\sigma$ be the supremum of topologies $\tau$ and $\sigma$;
it is generated by the base $\{U\cap V:U\in\tau, V\in\sigma\}$.
Put $S=s^{-1}(0)\cup\{e/n:n\in\mathbb N\}$. Clearly, $\langle S\rangle=L$,
and, using that $s(y)\le 1$ for each $y\in S$, we can easily show that
$S$ is closed and discrete in $(L,\sigma)$ and so in $(L,\tau\vee\sigma)$.
Since $T$ is dense in $L$, by~\cite[Proposition 5.36]{BR2020}, the topologies
$\tau$ and $\sigma$ are cowide.
Since $L=T-T$, by~\cite[Proposition 5.35]{BR2020}, the topology $\sigma$ is wide.
Thus by ~\cite[Lemma 5.31]{BR2020}, $(\tau\vee\sigma)_{sr}=\tau_{sr}=\tau$.
That is the semiregularization of the group $(L,\tau\vee\sigma)$ is $(L,\tau)$, which
has no suitable set.
\end{proof}

\subsection{Countable paratopological groups}\label{sec:countable}
Recall that a group $G$ endowed with a topology is called {\it left topological} provided
each left shift $x\mapsto gx$, $g\in G$, is a continuous map. Clearly,
each paratopological group is left topological.

By Theorem 2.2 from~\cite{CMRS1998}, each countable Hausdorff topological group is generated by
a closed discrete set. Guran generalized this result to Hausdorff left topological groups,
announced it in~\cite{G2003} (where groups are supposed to be Hausdorff)
and presented it at a seminar; Banakh presented a proof in Theorem 13.3 of ~\cite{BP2003}
\footnote{We clarified these details via a personal communication with Guran.}.
Since an arbitrary locally finite group $G$, which is not finitely generated,
endowed with a cofinite topology
(consisting of the empty set and all subsets $A$ of $G$ such that $G\setminus A$ is finite)
is a left topological group without a suitable set, the following question remains unanswered.

\begin{question}\label{qc}
Does each countable paratopological group have a suitable set?
\end{question}

\section{Paratopological groups which are compact-like spaces}\label{sec:compactlike}

Different classes of compact-like spaces and relations between them
provide a well-known investigation topic in general topology, see, for instance, basic \cite[Chap.
3]{Eng} and general \cite{Vau}, \cite{Ste}, \cite{DRRT}, \cite{Mat}, \cite{Lip} works.

In the present paper we consider the following compact-like spaces.
A space $X$ is {\it countably compact}, if each its infinite subset $A$ has an accumulation point $x$
(the latter means that each neighborhood of $x$ contains infinitely many points of the set $A$),
or, equivalently, if each countable open cover of $X$ has a finite subcover.
A space is {\it feebly compact}
if each locally finite family of its nonempty open subsets of the space is
finite. A space is pseudocompact if and only if it is Tychonoff and feebly compact,
see \cite[Theorem 3.10.22]{Eng}.
According to~\cite{GutRav}, a
space $X$ is {\it sequentially pracompact} if it contains a dense set $D$ such that
each sequence of points of the set $D$ has a convergent subsequence.
Each countably compact and each sequentially pracompact space is feebly compact.

In Sections 1.4 and 1.5 of~\cite{BR2020} are listed different classes of
compact-like spaces and paratopological groups and considered relations between
these classes. Remark that an investigation of compact-like paratopological groups is
motivated by a problem on automatic continuity of inversion in a paratopological group.
An interested reader can find known results and references on this subject
in Section 5.1 of~\cite{Rav3} and Section 3 of the survey~\cite{Tka5},
in Introduction of~\cite{AlaSan} and in ~\cite{BR2020}.
In Section 1.6 of~\cite{BR2020} is provided a brief survey of the topic.

A cardinal $\kappa$ is called {\it measurable} if there is an $\aleph_{0}$-complete non-principal
ultrafilter on $\kappa$. It is consistent with ZFC that there are no measurable cardinals.
Moreover, under $V=L$ all cardinals are non-measurable, see~\cite{Dra74}.

\begin{proposition} Let $\kappa$ be an infinite non-measurable cardinal and $\lambda=2^{\kappa}$.
Let $G$ be a paratopological group such that $d(G)\le\lambda$.
If the group $G^\lambda$ is not countably compact then it has a closed suitable set.
\end{proposition}
\begin{proof} The proof is similar to that of~\cite[Theorem 3.2]{DTT1999}.
\end{proof}

A group is {\it locally finite} if each its finitely generated subgroup is finite.
Clearly, each periodic Abelian group is locally finite.

\begin{proposition}\label{prop:CCnotCS}
Let $G$ be a countably compact infinite locally finite paratopological group without non-trivial convergent
sequences. Then $G$ has no suitable set.
\end{proposition}
\begin{proof}
Suppose for a contradiction that $S$ is a suitable set for $G$.
Since $G$ is locally finite, $S$ is infinite.
Since $G$ is countably compact and $S$ is a suitable set, $S\cup\{e\}$ is a countably compact
space with the only non-isolated point $e$. It follows that any sequence of distinct points
of $S$ converges to $e$, a contradiction.
\end{proof}

Proposition~\ref{prop:CCnotCS} is closely related with
the following question by van Douwen~\cite{vDou} from 1980,
which is, according to references from~\cite[p. 2]{HvMR-GS},
considered central in the theory of topological groups.

\begin{question} Is there an infinite countably compact topological group without non-trivial convergent sequences?
\end{question}

For a history of related results see the introductions of~\cite{Tom2019} and~\cite{HvMR-GS}. In particular,
in 1976 Hajnal and Juh\'asz in~\cite{HJ} constructed the first example of such a group
under the Continuum Hypothesis. In 2004 Garc\'ia-Ferreira, Tomita, and Watson in~\cite{G-FTW}
obtained such a group from a selective ultrafilter. In 2009 Szeptycki and Tomita in~\cite{ST}
showed that in the Random model there exists such a group, giving the first example
that does not depend on selective ultrafilters.
Finally, Hru\v s\'ak, van Mill, Ramos-Garc\'ia, and Shelah obtained a positive answer to
van Douwen's question in ZFC, see~\cite{HvMR-GS} for "a presentable draft of the proof".
The group built there is Boolean and at p.17 the authors remark that they do not know how
to modify (in ZFC) their construction to get a non-torsion example and formulate
the question about an existence of such a group in ZFC.

We denote by TT an axiomatic assumption on an existence of such a torsion-free Abelian
group. An example of such a group was constructed by Tkachenko~\cite{Tka}
under the Continuum Hypothesis. Later, the Continuum Hypothesis was weakened to the
Martin's Axiom for $\sigma$-centered posets by Tomita in~\cite{Tom1996}, for countable posets in
\cite{KTW}, and finally to the existence of continuum many incomparable selective
ultrafilters in \cite{MT}.

The proof of~\cite[Lemma 6.4]{BDG} implies that under TT there exists a
group topology $\tau$ on a free Abelian group $F$ generated by the set $\mathfrak c$
with the following property: for each countable infinite subset $M$ of the group $F$ there
exists an element $\alpha \in \overline M\cap\mathfrak c$ such that $M$ is contained in a
subgroup $\langle \alpha \rangle$ of $F$, generated by a set $\alpha$.
Moreover, in~\cite[Example 3.22]{BR2020}, a topology $\sigma\supset\tau$ is constructed on $F$
such that $(F,\sigma)$ is a countably compact non-regular paratopological group,
which is not a topological group and has the property mentioned above.
The latter implies that both $(F,\tau)$ and $(F,\sigma)$ have no suitable set.

\begin{example}\label{e1}
There exists a non-Hausdorff periodic countable sequentially pracompact
Abelian paratopological group $G\in\mathcal{S}_{cg}$.
\end{example}
\begin{proof} The group $G$ was constructed by Banakh in~\cite[Example 3.18]{BR2020} as follows.
For each positive integer $n$ let $C_n$ be the set $\{0,\dots,n-1\}$
endowed with a binary operation $\oplus$ such that $x\oplus y\equiv x+y \mod n$ for any $x,y\in
C_n$, that is $C_n$ is a cyclic group.
Let $G=\bigoplus\limits_{n=1}^\infty C_n$ be a direct sum of the groups $C_n$.
Let $\mathcal F$ be a family of all non-decreasing unbounded functions from $\omega\setminus\{0\}$ to
$\omega$. For each $f\in\mathcal F$ put
$$O_f=\{0\}\cup\{(x_n)\in G:\exists n\in\mathbb N\; (0<x(n)<f(n)\mbox{ and }\forall m>n\;x(m)=0)\}.$$
It is easy to check that the family $\{O_f:f\in\mathcal F\}$ is a base
at the zero of a semigroup topology $\tau$ on $G$.

For each positive integer $m\ge 2$ consider a function
$\delta_m\in G$ such that $\delta_m(x)$ equals $1$, if $x=m$, and equals $0$, otherwise.
We claim that $S=\{0\}\cup\{-\delta_{m}: m\ge 2\}$ is a closed suitable set.
Indeed, since $G$ is a direct sum of cyclic groups $C_m$ for $m\ge 2$ and $\delta_m$ is a
generator of $C_m$, we have $\langle S\rangle=G$. Let
$f: \omega\setminus\{0\}\rightarrow\omega$ be a function such that $f(1)=0$ and
$f(i)=i-2$ for any $i\geq 2$. It is easy to see that a set $(x+O_{f})\cap S$ is finite
for any $x\in G$. Since $G$ is a $T_1$ space, it follows that $S$ is a closed discrete subset of $G$.
\end{proof}

The following result complements Proposition~\ref{prop:CCnotCS}.

\begin{example}\label{e2}
There exists a feebly compact non-countably compact Baire Hausdorff paratopological group
$G\in \mathcal{S}_{cg}$ such that all its countable subsets are closed.
\end{example}
\begin{proof}
Let $G$ be the feebly compact non-countably compact Baire Hausdorff Abelian paratopological group
with all countable subsets closed constructed by Sanchis and Tkachenko in the proof of~\cite[Theorem
2]{ST2012}. They also constructed an open subsemigroup $C$ of the group $G$
such that $C^{-1}$ is a closed discrete subspace of $G$ and $G=CC^{-1}$, so $G\in
\mathcal{S}_{cg}$.
\end{proof}

The following result extends Proposition 2.7 from~\cite{DTT1999}.

\begin{proposition}\label{p0new} Let $G$ be a countably compact paratopological group
with a suitable set, $H$ be a Hausdorff paratopological group, and $f:G\to H$ be a
continuous homomorphism with dense image. Then $H$ has a suitable set.
\end{proposition}
\begin{proof} A group $\langle f(S)\rangle=f(\langle S\rangle)$ is dense in $H$.
We claim that $f(S)\cup\{e\}$ is a compact space with all points of $S'=f(S)\setminus\{e\}$
isolated. Indeed, let $V$ be any open neighborhood of $e$ in $H$.
Suppose for a contradiction that a set $f(S)\setminus V$ is infinite.
Then a set $S\setminus f^{-1}(V)$ is infinite too. Since the space $G$ is countably compact,
the set $S\setminus f^{-1}(V)$ has an accumulation point $x\in G\setminus f^{-1}(V)$,
which is impossible since $x\ne e$. If $f(S)=\{e\}$ then $H=\overline{\langle f(S)\rangle}=\{e\}$.
Otherwise $S'$ is a suitable set for $H$, because
$\langle f(S)\rangle=\langle S'\rangle$, $S'\cup\{e\}=f(S)\cup\{e\}$ is a compact subset of a
Hausdorff space $H$ and so closed in it, and all points of $S'$ are isolated in $S'$.
\end{proof}

\section{Saturated paratopological groups}\label{sec:saturated}

Following Guran we call a paratopological group {\it saturated} if for each
neighborhood $U$ of the identity of $G$ its inverse $U^{-1}$ has
non-empty interior in $G$ ~\cite{BR2018}.

By {\it the group reflexion} $G^\flat=(G,\tau^\flat)$ of a
paratopological group $(G,\tau)$ we understand the group $G$
endowed with the strongest topology $\tau^\flat\subset\tau$
turning $G$ into a topological group. This topology admits a
categorial description: $\tau^\flat$ is a unique topology on $G$
such that\begin{itemize}
\item $(G,\tau^\flat)$ is a topological group; \item the identity
homomorphism $\operatorname{id}:(G,\tau)\to(G,\tau^\flat)$ is continuous; \item
for each continuous group homomorphism $h:G\to H$ into a
topological group $H$, the homomorphism $h\circ \operatorname{id}^{-1}:G^\flat\to
H$ is continuous.
\end{itemize}
Observe that the group reflexion of the Sorgenfrey line $\mathbb S$ is
the usual real line $\mathbb R$.~\cite{BR2018}.

By~\cite[Proposition 3]{BR2001}, for a saturated paratopological group
$(G,\tau)$, a base at the identity of the topology $\tau^\flat$
consists of the sets $UU^{-1}$, where $U$ runs over open neighborhoods of $e$ in $G$.
It follows (see also~\cite[Corollary 2]{BR2001}) that if $G$ is Hausdorff then
its group reflexion $G^\flat$ is Hausdorff as well.
Theorem 5 from ~\cite{BR2001} implies the following

\begin{lemma}\label{l:Th5BR2001} Let $G$ be a saturated paratopological group. Then
each non-empty open subset of $G$ contains a non-empty open subset of a (non-necessary $T_1$)
group $G^\flat$.
\end{lemma}

\begin{proposition}\label{t4}
Let $G$ be a saturated Hausdorff paratopological group and $S$ be a (closed) suitable
set for a group $G^{\flat}$. Then $S$ is a (closed) suitable set for a group $G$.
\end{proposition}
\begin{proof} Since the identity map from $G$ to $G^\flat$ is continuous,
$S$ is a (closed) discrete subspace of $G$ and $S\cup \{e\}$ is closed in $G$.
By Lemma~\ref{l:Th5BR2001}, a dense in $G^\flat$ set $\langle S\rangle$ is also dense in $G$.
\end{proof}

\subsection{Precompact paratopological groups}
Precompact paratopological groups are both compact-like and saturated
(see \cite[Proposition 3.1]{Rav2} for the proof of the latter).
Recall that a paratopological group $G$ is {\it precompact}
if for each neighborhood $U$ of the identity of $G$,
there exists a finite subset $F$ of $G$ such that $FU=G$.
By Proposition 2.1 from~\cite{Rav2} a paratopological group $G$ is precompact iff
for each neighborhood $U$ of the identity of $G$,
there exists a finite subset $F$ of $G$ such that $UF=G$.

Precompact topological groups are exactly subgroups of compact topological groups,
see, for instance,~\cite[Corollary 3.7.17]{AT},
which implies that a subgroup of a precompact topological group is precompact.
But a subgroup of a precompact paratopological group can fail to be precompact.
Moreover, by Corollary 5 from~\cite{BR2003}, subgroups of precompact Hausdorff
paratopological groups are exactly so-called Bohr-separated Hausdorff paratopological groups.
The latter class includes, for instance, all locally compact Hausdorff
topological groups and all locally convex Hausdorff linear topological spaces
(or more generally all locally quasi-convex Abelian Hausdorff topological groups,
see \cite{Aus} or \cite{Ba}).

On the other hand, a dense subgroup of a precompact paratopological group is precompact, see
\cite[Lemma 1]{T2013}. Also the following simple proposition holds

\begin{proposition} An open subgroup $H$ of a precompact paratopological group $G$
is precompact.
\end{proposition}
\begin{proof}
Pick any open neighborhood $U$ of $e$ in $H$. Since $H$ is open in $G$, the set $U$ is open in $G$,
hence there exists a finite subset $F$ of $G$ such that $FU=G$. Since
$(F\setminus H)U\cap H\subset (F\setminus H)H\cap H=\varnothing$, $H\subset (F\cap H)U$.
\end{proof}

\begin{proposition}\label{l15+}
Let $G$ be a non-feebly compact precompact paratopological group with a
countable dense subgroup $P$. Then there exists a subset $L$ of $P$ such that $L$
is discrete and closed in $G$ and $\langle L\rangle=P$. In particular,
$L$ is suitable for $G$, and thus $G\in\mathcal{S}_{c}$.
\end{proposition}
\begin{proof}
Since the space $G$ is not feebly compact, there is an infinite locally finite sequence
$\{V_n\}$ of non-empty open subsets of $G$.
By~\cite[Proposition 3.1]{Rav2}, the group $G$ is saturated.
By Lemma~\ref{l:Th5BR2001}, for each $n$ there
exists a non-empty open in $G^\flat$ set $W_n\subset V_n$.
Put $U_n=\bigcup_{k\ge n} W_k$. Then $\{U_n\}$ is a non-decreasing
family of non-empty open in $G^\flat$ sets which is locally finite in $G$.
Let $\{x_{n}: n\in\omega\}$ be an enumeration of elements of $P$.
As in the proof of~\cite[Lemma 3.5]{DTT1999} we can construct
by induction an increasing sequence $\{L_{k}: k\in\omega\}$
of finite subsets of $P$ such that for each $k\in\omega$ we have
$x_{k}\in \langle L_{k}\rangle$, $L_{k+1}\setminus L_{k}\subset U_{k}$, and
$G=\langle L_{k}\rangle U_{k}$. Since the family $\{U_n\}$ is
locally finite in $G$, it follows that $L=\bigcup_{n\in\omega} L_n$
is a closed discrete subset of $G$ and $\langle L\rangle=P$.
\end{proof}

\subsection{Non-precompact paratopological groups}

For a Hausdorff paratopological group $G$ the {\it Hausdorff number} of $G$,
denoted by $Hs(G)$, is the minimum cardinal number $\kappa$
such that for every neighborhood $U$ of $e$
there exists a family $\gamma$ of neighborhoods of $e$ such that $|\gamma|\le\kappa$ and
$\bigcap_{V\in\gamma} VV^{-1}\subset U$, see~\cite{Tka2009}.

\begin{proposition}\label{c5} Let $G$ be a Hausdorff locally separable saturated (non-precompact)
pa\-ra\-to\-po\-lo\-gi\-cal group such that $Hs(G)\psi(G)\le\omega$. Then $G$ has a (closed)
suitable set.
\end{proposition}
\begin{proof}
Since a base at the identity of the topology $\tau^\flat$
consists of the sets $UU^{-1}$, where $U$ runs over open neighborhoods of $e$ in $G$,
$\psi(G^\flat)\le Hs(G)\psi(G)\le\omega$.
Since $G$ is saturated and locally separable, by Lemma~\ref{l:Th5BR2001}
the group $G^{\flat}$ is locally separable and if the group $G$ is non-precompact then
the group $G^{\flat}$ is non-precompact as well.
By Theorem 5.14 (Corollary 5.9) from~\cite{CMRS1998}, the group $G^{\flat}$ has a (closed) suitable set.
By Proposition~\ref{t4}, $G$ has a (closed) suitable set.
\end{proof}

\begin{corollary}\label{c5+} Each Hausdorff locally separable saturated first countable
(non-pre\-com\-pact) paratopological group has a (closed) suitable set.
\end{corollary}

\begin{example} The Sorgenfrey line $\mathbb{S}$ is a Hausdorff first countable
hereditarily Lindel\"of (by \cite[Exercise 3.8.A.c]{Eng})
saturated paratopological group.
Any two-element subset $\{x,y\}$ of $\mathbb R$ such that $x/y$ is irrational
is a suitable set both for $\mathbb R$ and $\mathbb S$.
By Proposition~\ref{t2.2}, each dense subgroup of $\mathbb{S}$ has a countable closed suitable set.
However, a subgroup $\mathbb Q$ of $\mathbb{S}$ or $\mathbb{R}$ has no finite suitable set, because
each finitely generated subgroup of $\mathbb{Q}$ is discrete in $\mathbb R$.
\end{example}

\begin{question}
Does every Hausdorff locally separable first-countable paratopological group $G$ have a suitable set?
\end{question}

Recall that in~\cite{G2003} is announced that every Hausdorff separable non-precompact
paratopological group $G$ has a closed suitable set.

\begin{remark}
By \cite[Corollary 3.14]{DTT1999}, a free (Abelian) topological group over a metrizable space $X$
has a closed suitable set. On the other hand,
Theorem 3.8 from~\cite{CMRS1998} implies that a (Markov) free topological group
$F(\beta D)$ over a \v{C}ech-Stone compactification of an uncountable discrete space $D$
has no suitable set.
By Example 2.13 from~\cite{DTT1999}, there exists a dense open pseudocompact subspace
$Y$ of $\beta\mathbb N\setminus\mathbb N$
such that the free Abelian topological group $A(Y)$ over the space $Y$ has a closed suitable set.
Therefore, the free Abelian topological group $A(\beta\mathbb N\setminus\mathbb N)$
over $\beta\mathbb N\setminus\mathbb N$ does not have a suitable set, but it contains a dense subgroup
$A(Y)\in\mathcal S_c$.

Usually, a free (Abelian) paratopological group over a space $X$ (see, for instance, ~\cite{PN2006}
for the definitions) is much more asymmetric than a free topological group over $X$.
In the proof of~\cite[Proposition 3.5]{PN2006}
Pyrch and Ravsky showed that for any non-empty space $X$, a set $-X$ is a closed suitable
set for the (Markov) free Abelian paratopological group $A_p(X)$. A similar result holds for
the (Markov) free paratopological group over $X$, which, by~\cite[Theorem 4.13]{EN2012},
contains $X^{-1}$ as a closed subset and a discrete subspace. In particular,
if $X=\beta D$ then, by Proposition 2.2 from~\cite{PN2006},
the group reflection of this group is topologically isomorphic
to a free topological group over $X$, which has no suitable set by
Theorem 3.8 from~\cite{CMRS1998}.
\end{remark}

\begin{example}\label{ex:GflatwoSS}
There exists a zero-dimensional saturated paratopological group $G$
which is a sum of two of its closed discrete subsets such that the group reflexion
$G^{\flat}$ has no suitable set.
\end{example}
\begin{proof}
According to~\cite[Theorem 2.4.a]{DTT1999} there exists a non-separable Lindel\"of
locally convex linear topological space $L$ of countable pseudocharacter.
Put $G=L\times\mathbb R$.
Let $\mathscr{B}$ be a local base at the identity of the group $L$
consisting of closed convex sets $U$ such that $U=-U$.
For each natural $n$ and each $U\in\mathscr{B}$ put
${\check U}_n=\{(x,t)\in L\times\mathbb R:0\le t< 1/n, x\in tU\}$.
It is easy to check that the family
$\check{\mathscr{B}}=\{{\check U}_n:U\in\mathscr{B}, n\in\mathbb N\}$
satisfies Pontrjagin conditions (see, for instance,~\cite[Proposition 1.1]{Rav})
for a local base at the identity of a semigroup topology on the group $G$.
Denote this topology by $\tau$.

We claim that  $\check{\mathscr{B}}$ consists of closed subsets of the
group $(G,\tau)$, so the latter is zero-dimensional. Indeed, let
$U$ be any closed convex neighborhood of the identity of $L$, $n$ be any natural number,
and $(y,s)$ be any point of $G\setminus {\check U}_n$.
If $s\ge\tfrac 1n$ then $((y,s)+{\check U}_n)\cap {\check U}_n=\varnothing$
If $s<0$ then $((y,s)+{\check U}_m)\cap {\check U}_n=\varnothing$ for each $m$ with $s+\tfrac 1m<0$.
Assume that $0\le s<\tfrac 1n$. Since $y\not\in {\check U}_n$, $y\not\in sU$. Since
the set $sU$ is closed in $L$, there exists $m\in\mathbb N$
such that $y\not\in\left(s+\tfrac 1m\right)U$.
We claim that $((y,s)+{\check U}_{2m})\cap {\check U}_n=\varnothing$.
Indeed, suppose for a contradiction that there exist $u,u'\in U$ and $0\le t<\tfrac 1n$,
$0\le t'<\tfrac 1{2m}$, such that $(y,s)+(t'u',t')=(tu,t)$. Clearly, $t+t'\ne 0$.
Since $U=-U$, we have
$$y=tu-t'u'=(t+t')\frac{tu+t'(-u)}{t+t'}\in (s+2t')U\subset \left(s+\frac 1m\right)U,$$
a contradiction.

It is easy to see that the group $(G,\tau)$ is saturated and
its group reflexion $G^\flat$ is the Cartesian product $L\times\mathbb R$.
So $G^\flat$ is a non-separable space of countable pseudocharacter.
Since the space $\mathbb R$ is $\sigma$-compact, the space $G^\flat$ is Lindel\"of,
see~\cite[Corollary 3.8.10]{Eng}. By Proposition~\ref{prop:DPTGSS}, $G^\flat$ has no suitable set.

Pick any non-zero vector $x_0\in L$. The line $\ell=\mathbb Rx_0$ is closed in $G^\flat$
and so in $(G,\tau)$. It is easy to check that
for any $x\in\ell$ and any $x_0\not\in U\in\mathscr{B}$
we have $(x+{\check U}_1)\cap\ell=\{x\}$, so $\ell$ is a discrete subset of $(G,\tau)$.

A group $L'=L\times \{0\}$ is closed in $G^\flat$ and so in $(G,\tau)$.
Clearly, for any $x\in L'$ and any $U\in\mathscr{B}$
we have $(x+{\check U}_1)\cap L'=\{x\}$, so $L'$ is a discrete subset of $(G,\tau)$.

Finally, we have $G=\ell+L'$. In particular, $\ell\cup L'$ is a closed discrete generating
subset of $G$.
\end{proof}

A subset $Y$ of a space $X$ is {\it strictly $\sigma$-discrete} (in $X$)
if $Y$ is a countable union of closed discrete subsets of $X$~\cite{DTT1999}.

We can adapt Theorem 3.6.I from ~\cite{DTT1999} to saturated paratopological groups as follows.

\begin{proposition}\label{ttt+}
Every non-feebly compact non-precompact saturated paratopological group $G$ with
a dense strictly $\sigma$-discrete subspace has a closed suitable set.
\end{proposition}
\begin{proof}
By Lemma~\ref{l:Th5BR2001},
the group $G^\flat$ is non-precompact too. That is $G^\flat$
has a neighborhood $U$ of the identity such that $FU\ne G$ for any finite subset
$F$ of $G$. Pick an arbitrary symmetric neighborhood $V$ in $G^\flat$ of
the identity such that $V^4\subset U$. It can be shown
(see, for instance,~\cite[Remark 5.4]{CMRS1998})
that the family $\{aV:a\in A\}$ is discrete in $G^\flat$ for some countably infinite subset $A$ of $G$.
For each $a\in A$ let $H_a$ be a closed discrete subspace of $G$ such that
$H=\bigcup_{a\in A} H_a$ is dense in $G$. Then $F_a=H_a\cap V \cup\{e\}$ is
a closed discrete subspace of $G$ for each $a\in A$
and a set $F=\bigcup_{a\in A} F_a$ is dense in $V$.
Since $\{aF_a: a\in A\}$ is a locally finite family
of closed discrete subsets, by Lemma 6.1 from~\cite{CMRS1998}
its union $S$ is closed and discrete.
Since $\langle F\rangle$ is dense in $V$ and
$\langle S \rangle\supset F$, $\langle S \rangle$ is dense in
the subgroup $\langle V\cup A\rangle$ of $G^\flat$.
So $S$ is a closed suitable set for the open subgroup
$\langle V\cup A\rangle$ of $G$. By Proposition~\ref{t6}, the group $G$ has a closed
suitable set as well.
\end{proof}

A regular space is {\it a $\sigma$-space} provided it has $\sigma$-discrete
(equivalently, $\sigma$-locally finite) network, see~\cite[4.3]{Gru}.
If a topological group is a $\sigma$-space then by Corollary 3.10 from~\cite{DTT1999}
it has a suitable set. This suggests the following

\begin{question}\label{qd}
Does every regular paratopological group which is a $\sigma$-space has a suitable set?
\end{question}

The following proposition provides a partial affirmative answer to Question~\ref{qd}.

\begin{proposition}\label{qdpart+} Every non-precompact saturated paratopological group $G$
which is a $\sigma$-space has a suitable set.
\end{proposition}
\begin{proof} If $G$ is non-feebly compact then $G$ has a suitable set by
Proposition~\ref{ttt+}. If $G$ is feebly compact then $G$ is a topological group
by Proposition 3.15 from~\cite{BR2020} and so $G$ has a suitable set by
Corollary 3.10 from~\cite{DTT1999}.
\end{proof}

An {\it almost topological group} is a paratopological group $(G,\tau)$
admitting a Hausdorff group topology $\gamma\subset\tau$ and
a local base $\mathscr{B}$ at the identity $e$ of the group $(G, \tau)$
such that the set $U\setminus \{e\}\in\gamma$ for each $U\in\mathscr{B}$.
In this case we say that $G$ is an almost topological group with
a {\it structure} $(\tau, \gamma, \mathscr{B})$. \cite{Fer2012}

\begin{remark} If in the proof of Example~\ref{ex:GflatwoSS} we take
as $\mathscr{B}$ a local base at the identity of the group $L$
consisting of {\it open} convex sets, we construct
a Hausdorff saturated paratopological group $(G,\tau)$
which is a sum of two of its closed discrete subsets such that
$(\tau,\tau_\flat,\check{\mathscr{B}})$ is a structure of
an almost topological group on $G$, but
the group reflexion $G^{\flat}=(G,\tau_\flat)$ has no suitable set.
\end{remark}

\section{The permanence properties}\label{sec:perm}

\subsection{Open subgroups}
Similarly to the proof of \cite[Theorem 4.2]{CMRS1998}, we can show the following

\begin{proposition}\label{t6} If an open subgroup of a paratopological group $G$
has a (closed) suitable set, then $G$ has a (closed) suitable set.
\end{proposition}

\begin{remark} (cf.,~\cite[Theorem 4.7]{CMRS1998}). Let $G$ be any paratopological group and $G'$ be the group $G$ endowed
with the discrete topology. Then the group $G\times G'$ is generated by a closed discrete
set $G\times \{e\}\cup \{(x,x):x\in G\}$ and $G$ is naturally isomorphic to an open
subgroup $G\times \{e\}$ of $G\times G'$.
\end{remark}

\subsection{Dense subgroups}\label{sec:perm2}
A subset $D$ of a space $X$ is called {\it strongly discrete} if each point $x\in D$
has a neighborhood $U_x$ such that the family $\{U_x:x\in D\}$ is {\em discrete},
that is each point of $X$ has a neighborhood intersecting at most one set of the family.
It is easy to see that each strongly discrete subset of a $X$ is discrete and closed in $X$.
A space is called {\it collectionwise Hausdorff} if each its closed discrete subset
is strongly discrete.
A subset $D$ of a space $X$ is called {\it sequentially dense} (in $X$) if each point $x\in X$
is a limit of a sequence of points of $D$. Clearly, each dense subset of a
Fr\'{e}chet-Urysohn space is sequentially dense.

A dense subgroup a topological group $G$ with a
suitable set need not have a suitable set even when $G$ is compact,
see \cite[Corollary 2.9]{DTT1999}.

\begin{proposition}\label{t2}
Let $G$ be a hereditarily collectionwise Hausdorff paratopological group with a suitable set $S$ and
$H$ be a sequentially dense subgroup of $G$.
Then $H$ has a suitable set of size at most $|S|\cdot\aleph_0$.
\end{proposition}
\begin{proof}
Since the space $G\setminus\{e\}$ is collectionwise Hausdorff then
each point $x\in S$ has a neighborhood $U_x\not\ni e$ in $G$ such that
the family $\{U_x:x\in S\}$ is discrete in $G\setminus\{e\}$.
If $x\in H$ then let $S_x=\{x\}$, otherwise
let $S_x$ be a sequence of distinct points of $U_x\cap H$, converging to $x$.
Since $\{S_x:x\in S\}$ is a locally finite family
of closed discrete subsets of the space $G\setminus\{e\}$,
by Lemma 6.1 from~\cite{CMRS1998} its union $S'$ is closed and discrete in $G\setminus\{e\}$.
Since $\overline{\langle S'\rangle}\supset S$ and $\overline{\langle S\rangle}=G$,
$\overline{\langle S'\rangle}=G$. It follows that $S'$ is a suitable set for $H$.
\end{proof}

Similarly we can prove the following

\begin{proposition}\label{t2.2}
Let $G$ be a collectionwise Hausdorff
paratopological group with a closed suitable set $S$ and
$H$ be a sequentially dense subgroup of $G$.
Then $H$ has a closed suitable set of size at most $|S|\cdot\aleph_0$.
\end{proposition}

\begin{corollary}\label{t2c}
Let $G$ be a metrizable paratopological group with a (closed) suitable set $S$ and
$H$ be a dense subgroup of $G$. Then $H$ has a (closed) suitable set of size at most $|S|\cdot\aleph_0$.
\end{corollary}
\begin{proof} It suffices to remark that $H$ is sequentially dense in $G$ and
$G$ is collectionwise normal, by Theorems 5.1.3 and 5.1.18 from ~\cite{Eng}.
\end{proof}

\begin{proposition}\label{t2.3}
Let $G$ be a regular collectionwise Hausdorff paratopological group of countable pseudocharater
with a suitable set $S$ and $H$ be a sequentially dense subgroup of $G$.
Then $H$ has a suitable set of size at most $|S|\cdot\aleph_0$.
\end{proposition}
\begin{proof}
Let $\{U_n:n\in\omega\}$ be a non-increasing sequence of open subsets of $G$ such
that $U_0=G$ and $\bigcap\{U_n\}=\{e\}$.
Let $n\in\omega$ be any number. Put $V_n=U_n\setminus \overline{U_{n+2}}$.
Since the space $G$ is collectionwise Hausdorff,
each point $x$ of a closed discrete subspace $S\cap V_n$ of $G$ has a neighborhood
$U_{x,n}\subset V_n$ such that the family $\{U_{x,n}:x\in S\}$ is discrete.
If $x\in H$ then let $S_{x,n}=\{x\}$, otherwise
let $S_{x,n}$ be a sequence of distinct points of $U_{x,n}\cap H$, converging to $x$.
It is easy to see that $\{S_{x,n}:n\in\omega,\, x\in S\cap V_n\}$ is a locally finite family
of closed discrete subsets of the space $G\setminus\{e\}$,
so, by Lemma 6.1 from~\cite{CMRS1998}, its union $S'$ is closed and discrete in $G\setminus\{e\}$.
Since $\overline{\langle S'\rangle}\supset S$ and $\overline{\langle S\rangle}=G$,
$\overline{\langle S'\rangle}=G$. It follows that $S'$ is a suitable set for $H$.
\end{proof}

\subsection{Products}\label{sec:perm3}
 Let $G$ be the product of a family $\Gamma=\{G_i\colon i\in I\}$ of paratopological groups.
The $\Sigma$-($\sigma$-) product of $\Gamma$ is a subgroup of $G$ consisting of all
elements with countably (finitely) many non-identity coordinates.

Similarly to the proof of \cite[Theorem 4.1]{CMRS1998}, we can show the following

\begin{proposition} Let $\Gamma$ be a non-empty family of paratopological groups
such that each member of it has a suitable set. Then the product of $\Gamma$
has a suitable set contained in the $\sigma$-product of $\Gamma$. It follows
that both the $\sigma$- and $\Sigma$-product of $\Gamma$ has a suitable set.
\end{proposition}

The proof of the following proposition is similar to that of~\cite[Lemma 3.1]{DTT1999}.

\begin{proposition}\label{l1} If a paratopological group $G$ contains a closed discrete set $A$
such that $|A|\geq d(G)$ then $G\times G\in \mathcal{S}_{c}$.
Moreover, if $|A|=|G|$ then $G\times G\in \mathcal{S}_{cg}$.
\end{proposition}

A {\it pseudocharacter $\psi(X)$} of a space $X$ is the smallest
infinite cardinal $\kappa$ such that any point of $X$ is an intersection of at most $\kappa$
open subsets of $X$ and {\it extent $e(X)$}
is the supremum of cardinalities of closed discrete subspaces of $X$.

The following proposition complements Proposition~\ref{l1} and its proof is
similar to the proof of~\cite[Lemma 2.3]{DTT1999}.

\begin{proposition}\label{prop:DPTGSS} If $G$ is a paratopological group with
a suitable set then $d(G)\leq \psi(G)e(G)$. Moreover,
if $G$ has a closed suitable set then $d(G)\le e(G)$.
\end{proposition}
\begin{proof} Since the second claim is obvious we prove only the first.
Let $S$ be a suitable set for $G$.
For each open neighborhood $U$ of the identity, $S\setminus U$ is a closed discrete
subspace of $G$, so $|S\setminus U|\le e(G)$.
Let $\gamma$ be a family of open subsets in $G$ such that
$\bigcap\gamma=\{e\}$ and $|\gamma|\le\psi(G)$.
Then $S\setminus\{e\}=\bigcup\{S\setminus U: U\in\gamma\}$, so $|S|\le \psi(G)e(G)$.
Since $S$ is a suitable set for $G$,
$\langle S\rangle$ is a dense subset of $G$ of size at most $\psi(G)e(G)$.
\end{proof}

\subsection{Extensions} The following question is a counterpart of Problem 3.12 from~\cite{DTT2000} for paratopological
groups.

\begin{question}
Let $G$ be a paratopological group and $H$ a closed normal subgroup of $G$.
If $H$ and $G/H$ have suitable sets, does $G$ have a suitable set?
\end{question}

{\bf Acknowledgments.} The authors thank to Igor Guran for consultations
and to Taras Banakh for sharing with them the book~\cite{BP2003} and other help.

\end{document}